\documentclass[11pt,a4paper,%headinclude,footinclude
]{amsart}                 
\usepackage[T1]{fontenc}             
\usepackage[utf8]{inputenc}          \usepackage[english]{babel}       
\usepackage{graphicx}                      % 
\usepackage[font=small]{quoting}            %  
\usepackage{caption}  
\usepackage[colorlinks=true,linkcolor=blue]{hyperref}
\usepackage[top=1.3in,bottom=1.4in,left=1.0in,right=1.0in]{geometry}
\usepackage{verbatim}
\usepackage{color}
\usepackage{picture}
\usepackage{amsmath,amsthm,amsfonts,amssymb}
\usepackage{comment}
\usepackage{hyperref}
\hypersetup{colorlinks,linkcolor={blue},citecolor={blue},urlcolor={red}}  
\allowdisplaybreaks

\newdimen\AAdi%
\newbox\AAbo%
\def\AAk#1#2{\s_etbox\AAbo=\hbox{#2}\AAdi=\wd\AAbo\kern#1\AAdi{}}%
\def\AAr#1#2#3{\s_etbox\AAbo=\hbox{#2}\AAdi=\ht\AAbo\raise#1\AAdi\hbox{#3}}%
\font\tenmsb=msbm10 at 12pt
\font\sevenmsb=msbm7 at 8pt
\font\fivemsb=msbm5 at 6pt
\newfam\msbfam
\textfont\msbfam=\tenmsb
\scriptfont\msbfam=\sevenmsb
\scriptscriptfont\msbfam=\fivemsb
\def\Bbb#1{{\tenmsb\fam\msbfam#1}}

\newtheorem{thm}{Theorem}[section] 
 
\newtheorem{lem}[thm]{Lemma} 
 
\newtheorem{cor}[thm]{Corollary} 
 
\newtheorem{prop}[thm]{Proposition}

\theoremstyle{remark}

\theoremstyle{definition}

\newcommand{\beq}{\begin{equation}}
	\newcommand{\eeq}{\end{equation}}
\newcommand{\beqr}{\begin{eqnarray}}
	\newcommand{\eeqr}{\end{eqnarray}}
\newcommand{\ba}{\begin{array}}
	\newcommand{\ea}{\end{array}}

\allowdisplaybreaks
\begin{document}

	\newtheorem{pro}{Proposition}
	\newtheorem{defi}{Definition}
	\newcommand{\noi}{\noindent}
	\newcommand{\dis}{\displaystyle}
	\newcommand{\mint}{-\!\!\!\!\!\!\int}

	\def \bx{\hspace{2.5mm}\rule{2.5mm}{2.5mm}} \def \vs{\vspace*{0.2cm}}
	\def\hs{\hspace*{0.6cm}}
	\def \ds{\displaystyle}
	\def \p{\partial}
	\def \O{\Omega}
	\def \o{\omega}
	\def \b{\beta}
	\def \m{\mu}
	\def \l{\lambda}
	\def\L{\Lambda}
	\def \ul{u_\lambda}
	\def \D{\Delta}
	\def \d{\delta}
	\def \k{\kappa}
	\def \s{\sigma}
	\def \e{\varepsilon}
	\def \a{\alpha}
	\def \tf{\tilde{f}}
	\def\cqfd{%
		\mbox{ }%
		\nolinebreak%
		\hfill%
		\rule{2mm} {2mm}%
		\medbreak%
		\par%
	}
	\def \pr {\noindent {\it Proof.} }
	\def \rmk {\noindent {\it Remark} }
	\def \esp {\hspace{4mm}}
	\def \dsp {\hspace{2mm}}
	\def \ssp {\hspace{1mm}}
	
	\def \u{u_+^{p^*}}
	\def \ui{(u_+)^{p^*+1}}
	\def \ul{(u^k)_+^{p^*}}
	\def \energy{\int_{\R^n}\u }
	\def \sk{\s_k}
	\def \mo{\mu_k}
	\def\cal{\mathcal}
	\def \I{{\cal I}}
	\def \J{{\cal J}}
	\def \K{{\cal K}}
	\def \OM{\overline{M}}

	\def\fk{{{\cal F}}_k}
	\def\M1{{{\cal M}}_1}
	\def\Fk{{\cal F}_k}
	\def\Fl{{\cal F}_l}
	\def\FF{\cal F}
	\def\Gk{{\Gamma_k^+}}
	\def\n{\nabla}
	\def\uuu{{\n ^2 u+du\otimes du-\frac {|\n u|^2} 2 g_0+S_{g_0}}}
	\def\uuug{{\n ^2 u+du\otimes du-\frac {|\n u|^2} 2 g+S_{g}}}
	\def\sku{\sk\left(\uuu\right)}
	\def\qed{\cqfd}
	\def\vvv{{\frac{\n ^2 v} v -\frac {|\n v|^2} {2v^2} g_0+S_{g_0}}}
	\def\vvs{{\frac{\n ^2 \tilde v} {\tilde v}
			-\frac {|\n \tilde v|^2} {2\tilde v^2} g_{S^n}+S_{g_{S^n}}}}
	\def\skv{\sk\left(\vvv\right)}
	\def\tr{\hbox{tr}}
	\def\pO{\partial \Omega}
	\def\dist{\hbox{dist}}
	\def\RR{\Bbb R}\def\R{\Bbb R}
	\def\C{\Bbb C}
	\def\B{\Bbb B}
	\def\N{\Bbb N}
	\def\Q{\Bbb Q}
	\def\Z{\Bbb Z}
	\def\PP{\Bbb P}
	\def\EE{\Bbb E}
	\def\F{\Bbb F}
	\def\G{\Bbb G}
	\def\H{\Bbb H}
	\def\SS{\Bbb S}\def\S{\Bbb S}
	
	\def\lcf{{locally conformally flat} }

	\def\circledwedge{\setbox0=\hbox{$\bigcirc$}\relax \mathbin {\hbox
			to0pt{\raise.5pt\hbox to\wd0{\hfil $\wedge$\hfil}\hss}\box0 }}
	%\ignorespaces%
	
	\def\sss{\frac{\s_2}{\s_1}}

    \def\intsn{\int_{{\mathbb S}^n}}
	
	\numberwithin{equation} {section}

	\date{}
	\title{ The $\s_k$-Yamabe problem revisited}
\keywords{ $\sigma_k$-Yamabe problem, fully nonlinear equation, Yamabe minimizer }
\subjclass[2020]{Primary 53C21; Secondary 35J60, 58J05}
	%[]
	\author{Yuxin Ge}
	\address{Institut de Math\'ematiques de Toulouse,\\
		Universit\'e Paul Sabatier,\\
		118, route de Narbonne,\\
		31062 Toulouse Cedex, France}
	\email{yge@math.univ-toulouse.fr}
	
		\author{Guofang Wang}%\thanks{The second named author is partly supported by SFB/TR71 of DFG}
	\address{ Albert-Ludwigs-Universit\"at Freiburg,
		Mathematisches Institut,
		Eckerstr. 1,
		D-79104 Freiburg, Germany}
	\email{guofang.wang@math.uni-freiburg.de}
	
	\author{Wei Wei} 
\address{School of Mathematics, Nanjing University, Nanjing 210093, P.R. China}
\email{wei\_wei@nju.edu.cn}
	
	\begin{abstract}
		In this paper we revisit the $\s_k$-Yamabe problem on $M^n$, namely, \emph{finding a conformal metric with constant $\s_k$-scalar curvature}. We prove that on a closed manifold $(M,[g_0])$ with positive Yamabe constant $Y_1(M,[g_0])>0$, the $\s_2$-Yamabe constant
        \[
  Y_2(M,[g_0]) := \inf_{g\in [g_0],\, R_g>0}\frac{\int_M \s_2(g)\,dvol(g)}{\mathrm{vol}(g)^{\frac{n-4}{n}}}
        \]
        is achieved by a conformal metric $g\in[g_0]$, which in particular solves the $\s_2$-Yamabe problem, assuming $Y_2(M,[g_0])>0$. As a consequence, for any $(M,g_0)$ with $Y_1(M,[g_0])>0$ and $Y_2(M,[g_0])>0$ one has
        \[
     \inf_{g\in [g_0],\, R_g>0}\frac{\int_M \s_2(g)\,dvol(g)}{\mathrm{vol}(g)^{\frac{n-4}{n}}}
     = \inf_{g\in [g_0],\, R_g>0,\, \s_2(g)>0}\frac{\int_M \s_2(g)\,dvol(g)}{\mathrm{vol}(g)^{\frac{n-4}{n}}}.
        \]
        We also show that these conclusions can fail if the condition $R_g>0$ is removed.
	\end{abstract}

	\maketitle
	\section{Introduction}
	
	Let $(M,g_0)$ be a compact Riemannian manifold with metric $g_0$, and let $[g_0]$ denote the conformal class of $g_0$.
	Let $A_g$ be the Schouten tensor of a metric $g$, defined by
	\[
	A_g=\frac{1}{n-2}\left(Ric_g-\frac{R_g}{2(n-1)}\,g\right).
	\]
	Here $Ric_g$ and $R_g$ denote the Ricci tensor and the scalar curvature of $g$, respectively.
	The role of $A_g$ in conformal geometry is reflected in the decomposition of the Riemann curvature tensor,
	\[
	Riem_g = W_g + A_g\circledwedge g,
	\]
	where $\circledwedge$ is the Kulkarni--Nomizu product; in particular, $g^{-1}\cdot W_g$ is conformally invariant.

	The {\it $\s_k$-scalar curvature} (or {\it $k$-scalar curvature}) is defined by
	\[
	\s_k(g):=\s_k\bigl(g^{-1}\cdot A_g\bigr),
	\]
	where locally $(g^{-1}\cdot A_g)^i{}_j = \sum_k g^{ik}(A_g)_{kj}$ and $\s_k$ denotes the $k$-th elementary symmetric function.
	Equivalently, for an $n\times n$ symmetric matrix $A$ with eigenvalues $\L=(\l_1,\dots,\l_n)$, we set $\s_k(A)=\s_k(\L)$.
	In particular, $\s_1(g)=\frac{1}{2(n-1)}R_g$ is a constant multiple of the scalar curvature.
	The quantity $\s_k(g)$, first studied by Viaclovsky \cite{ViacDuke}, is a natural generalization of scalar curvature.

	A central theme in differential geometry is to find ``good'' metrics within a given conformal class; metrics of constant $\s_k$-curvature provide one such notion.
	When $k=1$, this reduces to the classical Yamabe problem, namely the existence of a metric with constant scalar curvature,
	\begin{equation}
	\label{eq_Yamabe}
	R_g = \mathrm{const}.,
	\end{equation}
	in the conformal class $[g_0]$. 
This celebrated problem was solved by Schoen \cite{Schoen84}, building on work of Yamabe \cite{Yamabe}, Trudinger \cite{Trudinger}, and Aubin \cite{Aubin}; see the survey \cite{Lee-Parker}. In particular, they proved that the Yamabe constant
    \begin{equation}
        \label{Y1}
        Y_1(M,[g_0]) =\inf_{g\in[g_0]} \frac{\int_M R_g\, dvol(g)}{\mathrm{vol}(g)^{\frac{n-2}{n}}}
    \end{equation}
    is achieved by a conformal metric $g\in[g_0]$,  usually called a \emph{Yamabe metric}, and hence solves \eqref{eq_Yamabe}. In particular, when $Y_1(M,[g_0])>0$ the Yamabe functional may be minimized within the positive scalar curvature subclass:
    \beq
    \inf_{g\in[g_0]} \frac{\int_M R_g\,dvol(g)}{\mathrm{vol}(g)^{\frac{n-2}{n}}}
    =
    \inf_{g\in[g_0],\, R_g>0} \frac{\int_M R_g\,dvol(g)}{\mathrm{vol}(g)^{\frac{n-2}{n}}}.
    \eeq

    For $k\ge 2$, the (standard) $\s_k$-Yamabe problem asks whether there exists a conformal metric $g$ in the admissible class ${\cal C}_k([g_0])$ satisfying
	\begin{equation}\label{eq1}
		\s_k(g)=\mathrm{const},
	\end{equation}
	assuming that
	\begin{equation}\label{eq2}
		{\cal C}_k([g_0])\neq\emptyset.
	\end{equation}
	Here
	\[
	\mathcal C_k([g_0])=\bigl\{g\in [g_0]\,\big|\, \lambda(g^{-1}A_g)\in \Gamma_k^+\bigr\},
	\]
	and
\begin{equation}\label{eq0}
\Gamma_k^+=\Bigl\{\Lambda=(\l_1,\ldots,\l_n)\in \R^n\,\Big|\, \s_j(\Lambda)>0\ \text{for all } j\le k\Bigr\}
\end{equation}
	is the G\aa rding cone. The restriction \eqref{eq2} is natural: for $k\ge 2$ the equation \eqref{eq1} is fully nonlinear, and ellipticity holds precisely in the admissible region.

	Starting from the work of Viaclovsky \cite{ViacDuke} and Chang--Gursky--Yang \cite{CGY2002}, the $\s_k$-Yamabe problem has been studied extensively; see, for instance, \cite{GWDuke, GWCrelle, GVann, GVInvent, LLActa, LLCPAM2003, GWPisa, STW, TWCV, LiLuc}.
	The cases most relevant for the present paper are $k=2$ and the locally conformally flat setting.
	In these situations one can exploit a variational structure: the results of \cite{GWCrelle, GWPisa, STW} show that the (standard) Yamabe-type constant
    \beq
     \bar Y_k(M,[g_0]) := \inf_{g\in \mathcal{C}_k([g_0])}\frac{\int_M \s_k(g)\,dvol(g)}{\mathrm{vol}(g)^{\frac{n-2k}{n}}}
    \eeq
    is achieved by some $g\in \mathcal{C}_k([g_0])$, which then solves \eqref{eq1}.
In other words, in these cases the standard $\s_k$-Yamabe problem is solvable by a minimizer in $\mathcal{C}_k([g_0])$.

	In this paper we enlarge the admissible class and consider, for $k<\frac{n}{2}$ and $\mathcal{C}_{k-1}([g_0])\neq\emptyset$,
\beq
      Y_k(M,[g_0]) := \inf_{g\in \mathcal{C}_{k-1}([g_0])}\frac{\int_M \s_k(g)\,dvol(g)}{\mathrm{vol}(g)^{\frac{n-2k}{n}}}.
\eeq
This variant was first considered by Guan--Lin--Wang \cite{GLWMathZ}.

We can now state our main result.

\begin{thm}
	\label{mainthm}
	Let $n\ge 5$ and let $(M^n,g_0)$ be a compact Riemannian manifold with ${\cal C}_1([g_0])\neq\emptyset$ (equivalently, $Y_1(M,[g_0])>0$).
	If
	\begin{equation}\label{positive_case0}
		 Y_{2}(M,[g_0])>0,
	\end{equation}
	then $Y_2(M,[g_0])$ is achieved by a metric $g\in \mathcal{C}_2([g_0])$; in particular, $\mathcal{C}_2([g_0])\neq\emptyset$.
	Moreover, $\bar Y_{2}(M,[g_0])= Y_{2}(M,[g_0])$, i.e.,
\begin{equation}\label{eq0.3twoinvariantequ}
     \inf_{g\in [g_0],\, R_g>0 }\frac {\int _M \s_2(g)\, dvol(g)}{\mathrm{vol}(g)^{\frac{n-4}{n}}}
     =
     \inf_{g\in [g_0],\, R_g>0,\, \s_2(g)>0}\frac {\int _M \s_2(g)\, dvol(g)}{\mathrm{vol}(g)^{\frac{n-4}{n}}}.
\end{equation}
	Conversely, if $\mathcal{C}_2([g_0])\neq\emptyset$, then \eqref{positive_case0} holds.
\end{thm}

Theorem \ref{mainthm} answers a question raised by J.~Case. A natural follow-up is whether one can study the $\s_2$-Yamabe problem \eqref{eq1} on the entire conformal class $[g_0]$, without assuming $R_g>0$.
Our examples  indicate that the associated minimization problem is ill-posed in that generality: one always has
\[
\inf_{g\in[g_0]}\frac{\int_M \s_2(g)\,dvol(g)}{\mathrm{vol}(g)^{\frac{n-4}{n}}}=-\infty.
\]
See Section \ref{AppendixA} for details.

%%%%%%%%%%%%%

Next assume that $3\le k\le n/2$ and that $(M,g_0)$ is locally conformally flat.
In this setting we also solve the $\s_k$-Yamabe problem in the larger cone $\mathcal{C}_{k-1}([g_0])$.
	
	\begin{thm}
		\label{mainthm2} Let  $3\le k \le n\slash 2 $    and let $(M^n,g_0)$ be a compact  locally conformally flat Riemannian
		manifold with ${\cal C}_{k-1} {([g_0])}\not=\emptyset$.
		Assume that 
\begin{equation}\label{cond+k}
			Y_{k}(M,[g_0])>0.
		\end{equation}
		Then equation \eqref{eq1} admits a solution $g\in \mathcal C_k([g_0])$, which is a minimizer of $Y_k(M ,[g_0])$ in $\mathcal{C}_{k-1}([g_0]).$ In particular, 
		\begin{equation}\label{s+k}
			{\cal C}_k{([g_0])}\not=\emptyset.
		\end{equation}
		Moreover,
		\beq\label{eq0.3bis_a}
		Y_{k}(M,[g_0])=\bar Y_{k}(M,[g_0]).
		\eeq
		
       Condition \eqref{cond+k} is necessary for \eqref{s+k}.
	\end{thm}

The implication \eqref{cond+k}\,$\Rightarrow$\,\eqref{s+k} was proved in \cite{GLWMathZ} via degree theory for fully nonlinear elliptic equations. In dimensions $n=3$ and $n=4$, related implications were established in \cite{CGY2003, GVJDG, GLWJDG, CatinoDjali2010} and led to several geometric applications. The new point in Theorem~\ref{mainthm2} is the existence of a minimizer under the positivity assumption \eqref{cond+k}, which in particular yields \eqref{eq0.3bis_a}. Since the proof follows the same strategy as for Theorem~\ref{mainthm} and since we have more general results in the forthcoming paper \cite{GWW}, we omit the details.

    %%%%%%%%%%%%%%%%%%%%%%%%%%%%

The study of $\s_k$-curvature has led to striking geometric consequences. For instance, Chang--Gursky--Yang \cite{CGY2003} proved a conformal sphere theorem in dimension $4$; see also the more recent developments in \cite{Chang-Gursky-Zhang}. Related results in dimension $3$ were obtained in \cite{GLWJDG, CatinoDjali2010}. We also refer to \cite{GWAdv, Lai, GeLiLian} for further conformal invariants and applications.
	Applications to sharp Sobolev and Moser--Trudinger type inequalities have been developed in \cite{CYCPAM, GWDuke}; see also \cite{GWW}. Regarding existence theory, the standard $\s_k$-Yamabe problem (i.e., within the cone $\mathcal{C}_k$) is known to be solvable for $k\ge \frac{n}{2}$; see \cite{GVann, TWCV, LiLuc}. Many other cases remain open.

\medskip
\noindent{\it Organization of the paper.}
In Section~2 we describe our idea of proof. Then we introduce a Yamabe-type heat flow \eqref{eq2.1} and establish its basic properties, including monotonicity along the flow in Section 3. This flow arises as the gradient flow of a perturbed functional \eqref{func-e}, which corresponds to a subcritical equation.
Section~4 is devoted to the proof of the convergence of the flow. %Theorem~\ref{mainthm} via the heat flow \eqref{eq2.1}. 
The key $C^2$ estimates are established in Subsection~\ref{C2estimates-nonconformal-flow}, and uniform parabolicity (equivalently, preservation of the positivity of $\s_1$) is proved in Subsection~\ref{Uniform parabolicity for non-conformal flow}. In Section~5 we obtain \eqref{eq0.3twoinvariantequ} by proving the equivalence of the subcritical Yamabe constants $Y_{\varepsilon}(M,[g_0])=\bar Y_{\varepsilon}(M,[g_0])$. Finally, in Section~6 we construct examples of metrics showing what can go wrong without the positivity of the scalar curvature.

    \bigskip
    
	\noindent{\bf Acknowledgments.}
We would like to thank J.~Case for raising the question that motivated us to revisit the $\s_k$-Yamabe problem.
Part of this work was carried out while W.~Wei was visiting the University of Freiburg, supported by an Alexander von Humboldt research fellowship; she thanks the Institute of Mathematics at the University of Freiburg for its hospitality.
W.~Wei is also partially supported by NSFC (Grant Nos.~12571218, 12271244).

\section{The  idea of proof}\label{Section:Yambe type flow}

%We  our ideas by considering (formally) the following conformal flow. 

We explain our main ideas for proving Theorems~\ref{mainthm} and~\ref{mainthm2}.

\subsection{The quotient equation}

As in our previous work \cite{GLWJDG, GWCAG}, we begin with a key observation: an appropriate {\it quotient equation} has better structural properties than the corresponding ``pure'' $\s_k$-equation.

\begin{lem}[\cite{GLWJDG}]\label{lem2.3}
For $1<k\le n$, set $F(W)=\frac{\s_{k}}{\s_{k-1}}(W)$ and $F^{ij}=\frac{\partial F}{\partial w_{ij}}$, where $W=(w_{ij})$. Then:
		\begin{itemize}
			\item[1)]  The matrix $(F^{ij})(W)$ is positive semidefinite for $W\in \Gamma_{k-1}^+$
			and is positive definite for $W \in \Gamma_{k-1}^+\backslash {\cal R}_1$, where ${\cal R}_1$ is the set of symmetric matrices of rank $1$.
			\item[2)] The function $F$ is concave in the cone $\Gamma_{k-1}^+$.
			When $k=2$, for all $W\in \Gamma_1^+$ and for all $R=(r_{ij})\in {\cal S}_n$, we have
			\beq\label{eq3.1}
			\ba{rcl}
			\ds\vs \sum_{i,j,k,l}\frac {\p^2}{\p w_{ij}\p w_{kl}}\left(\frac {\s_2(W)}{\s_1(W)}\right)r_{ij}r_{kl}
			&=&\ds -\frac{ \sum_{ij}(\s_1(W)r_{ij}- \s_1(R)w_{ij})^2 }{ \s_1^3(W)}.
			\ea
			\eeq
		\end{itemize}
	\end{lem}

Lemma \ref{lem2.3} was discovered by Pengfei Guan in collaboration with Changshou Lin and was used crucially in \cite{GLWJDG, GWCAG} to study the quotient equation
\begin{equation}\label{eq:2.2}
     \frac {\s_2(g)-\nu} {\s_1(g)} =c,
\end{equation}
for a given $\nu>0$. Lemma \ref{lem2.3} implies that \eqref{eq:2.2} is elliptic whenever $\s_2(g)$ is positive, negative, or changes sign. We also note that the concavity in Lemma \ref{lem2.3} appeared earlier in work of Huisken--Sinestrari (1999) \cite{HuiskenSine}.

Recently, Lemma \ref{lem2.3} has been further studied by Guan--Zhang \cite{GuanZhang}, who extended this observation to more general quotient equations, including a Krylov-type equation. Motivated by these developments, we introduce a new fully nonlinear conformal flow and prove a more general Sobolev-type inequality in \cite{GWW}, extending the Sobolev inequalities proved by Ge--Wang \cite{GWAdv} and Ge--Wang--Xia \cite{GWX}.

Our main idea for Theorem~\ref{mainthm} is as follows. We modify \eqref{eq:2.2} slightly, but crucially, and consider
 \begin{eqnarray} \label{eq:2.3}
          \frac {\s_2(g)-r(g)} {\s_1(g)} =s. 
 \end{eqnarray}
 Here $r(g)$ denotes the average of $\s_2(g)$, i.e.,
 \begin{equation}
     r(g)= \frac{1}{\mathrm{vol}(g)} \int_M \s_2(g)\, dvol(g).
 \end{equation}
 If $r(g)>0$, then \eqref{eq:2.3} is elliptic and concave by Lemma~\ref{lem2.3}.
 If \eqref{eq:2.3} has a solution, then integrating \eqref{eq:2.3} and using the definition of $r(g)$ yields $s=0$, and hence
 \[\s_2(g)=r(g),\]
 which is the desired constant-$\s_2$ condition.
Therefore, we reduce the problem to proving existence for the quotient equation \eqref{eq:2.3}. To obtain a solution of \eqref{eq:2.3}, we use a flow approach (see the next subsection).

\section{A perturbed Yamabe-type flow}

	We now focus on the case $k=2$ and $n\ge 5$. Formally, a Yamabe-type flow associated with the quotient equation \eqref{eq:2.3} is
\beq\label{YamabeFlow}
	\frac {du}{dt}= -\frac 12 g^{-1}\frac d {dt }g =\frac {\s_2(g)- r(g)}{\s_{1}(g)}-s(g),
\eeq
but this flow appears difficult to analyze in general. For background on the classical Yamabe flow, see Brendle \cite{Brendle_Y_flow}. For our purposes it suffices to study a suitable perturbed flow, introduced next.
	
%\subsection{A perturbed Yamabe-type flow}
Denote
\[
{\cal F}_2(g):=\int_M \s_2(g)\,dvol(g).
\]
	Fix $\e\in[0,1)$ and consider the perturbed volume
	\[
	{\cal F}_{0,\e}(g)=\int_M e^{2\e u}\, dvol(g),
	\]
	for $g=e^{-2u}g_0$.
	A direct computation gives the first variation
	\[
	\ba{rcl}
	\ds\vs \frac d{dt}{\cal F}_{0,\e}(g)
	&=&\ds \frac{n-2\e}{2}\int_M e^{2\e u }\,g^{-1}\cdot \frac d {dt }g \, dvol(g).
	\ea
	\]
	In particular, for $\e=0$,
	\[
	\ba{rcl}
	\ds\vs \frac d{dt}{\cal F}_{0}(g)
	&=&\ds \frac{n}{2}\int_M g^{-1}\cdot \frac d {dt }g \,dvol(g).
	\ea
	\]

	We now introduce a flow that preserves ${\cal F}_{0,\e}$ and is monotone with respect to ${\cal F}_2$:
	\beq\label{eq2.1}
	\frac {du}{dt}= -\frac 12 g^{-1}\frac d {dt }g =e^{-2u}\frac {\s_2(g)- r_\e(g)e^{2\e u}}{\s_{1}(g)}+s_\e(g),
	\eeq
	where $r_\e(g)$ and $s_\e(g)$ depend only on $t$ and are determined by
	\beq\label{eq2.2}
	r_\e(g) :=\frac {\int_M \s_2(g)\, dvol(g)}{\int_M e^{2\e u}\, dvol(g)}
	\eeq
	and
	\beq\label{eq2.3}
	\int_M e^{2\e u}\left\{e^{-2u}\frac {\s_2(g)- r_\e(g)e^{2\e u}}
	{\s_{1}(g)}
	+s_\e(g)\right\}dvol(g)=0.\eeq
	In particular, when $\e =0$, we have
	
	\beq\label{eq2.1bis}
	\frac {du}{dt}= -\frac 12 g^{-1}\frac d {dt }g =e^{-2u}\frac {\s_2(g)- r(g)}{\s_{1}(g)}
	+s(g),\eeq
	where $r(g)$ and $s(g)$ depend only on $t$ and are determined by
	\beq\label{eq2.2bis}r(g) :=\frac {\int_M \s_2(g) \,dvol(g)}
	{\int_M 1\, dvol(g)}
	\eeq
	and
	\beq\label{eq2.3bis}
	\int_M \left\{e^{-2u}\frac {\s_2(g)- r(g)}
	{\s_{1}(g)}
	+s(g)\right\}\,dvol(g)=0.\eeq
    We remark that \eqref{eq2.1bis} is closer in spirit to a heat-type flow than to a Yamabe-type flow.
    
	It is also useful to note that stationary points of \eqref{eq2.1} solve a perturbed $\s_2$-equation (rather than a quotient equation). Indeed, if $g$ is stationary for \eqref{eq2.1}, then
    \beq
    e^{-2u}\frac {\s_2(g)- r_\e(g)e^{2\e u}}{\s_{1}(g)}=-s_\e(g).\eeq
    Multiplying by $\s_1(g)$, integrating over $M$, and using the definition of $r_\e$ yields $s_\e(g)=0$. Hence $g$ satisfies
	\begin{equation}\label{eq_a1}
		\s_2(g) -r_\e (g) e^{2\e u}=0.
	\end{equation}
	Thus \eqref{eq2.1} produces solutions of a perturbed $\s_2$-Yamabe equation, in contrast with the quotient-equation approach in \cite{GLWJDG}. We make this precise in the following lemma.

\begin{lem}\label{lem2.2}
For any $\e\ge 0$, flow \eqref{eq2.1} preserves ${\cal F}_{0,\e}$ and is non-increasing for $\mathcal{F}_2$. 

In particular, $r_\e$ is non-increasing along the flow.
Moreover, if the flow converges to a metric $g=e^{-2u}g_0$, then $g$ satisfies \eqref{eq_a1}.
\end{lem}
\begin{proof}
By the definition of $s_\e(g)$ (cf. \eqref{eq2.3}), the normalization in \eqref{eq2.1} preserves ${\cal F}_{0,\e}$.
Moreover, using \eqref{eq2.1} together with the definitions of $r_\e$ and $s_\e$, we compute
\beq\label{eq2.4}\ba{rcl}
-\frac{2}{n-4}\frac{d}{dt}{\cal F}_2(g) &=& -\ds\int_M \s_2(g)\, g^{-1}\!\cdot\! \frac{d}{dt}g\, dvol(g)\\
&=&- \ds\vs \int_M \bigl(\s_2(g)-r_{\e}(g)e^{2\e u}\bigr)\, g^{-1}\!\cdot\! \frac{d}{dt} g\, dvol(g)\\
&=& \ds\vs 2\int_M \bigl(\s_2(g)-r_{\e}(g)e^{2\e u}\bigr)\left(e^{-2u}\frac {\s_2(g)- r_\e(g)e^{2\e u}}{\s_{1}(g)}
+s_\e(g) \right) dvol(g)\\
&=& \ds\vs 2\int_M e^{2u}\s_{1}(g)\left(e^{-2u}\frac{\s_2(g)-r_\e(g)e^{2\e u}}{\s_{1}(g)}\right)^2 dvol(g).
\ea\eeq
This shows that ${\cal F}_2$ (and hence $r_\e$) is non-increasing along the flow.
If the flow converges, then the right-hand side of \eqref{eq2.4} tends to $0$, which forces \eqref{eq_a1}.
\end{proof}

For $\e\in[0,1)$, define
\begin{equation}\label{func-e}
\tilde {\cal F}_{2,\e}(g)=({\cal F}_{0,\e})^{-\frac{n-4}{n-2\e}}\int_M \s_2(g)\, dvol(g).
\end{equation}
We then introduce the associated subcritical Yamabe constants
\begin{equation}\label{subcritical Yamabe constant in larger cone}
Y_\e(M,[g_0])=\inf_{g\in {\cal C}_1([g_0])}\tilde {\cal F}_{2,\e}(g),
\end{equation}
whereas
\begin{equation}\label{subcritical Yamabe constant}
\bar Y_\e(M,[g_0])=\inf_{g\in {\cal C}_2([g_0])}\tilde {\cal F}_{2,\e}(g).
\end{equation}
Clearly $Y_\e(M,[g_0])\le \bar Y_\e(M,[g_0])$.
The discussion above also shows that the flow \eqref{eq2.1} decreases the functional $\tilde{\cal F}_{2,\e}(g)$.

If $g$ is a stationary point of the flow, then $g$ satisfies the perturbed equation
\beq\label{eq2.5}
\s_2(g)=c e^{2\e u},
\eeq
for some constant $c>0$. We view \eqref{eq2.5} as a subcritical approximation of the constant-$\s_2$ equation
\beq\label{eq2.5.2}
\s_2(g)=c.
\eeq

We will show that, for any sufficiently small $\e>0$, the infimum $Y_{\e}(M,[g_0])$ is achieved by a metric
$g_\e=e^{-2u_\e}g_0\in {\cal C}_2([g_0])$, which in particular solves \eqref{eq2.5}. Moreover, we will prove that
\[
\lim_{\e\to 0} Y_\e=\bar Y_{2,0}
\]
and that
\begin{equation}\label{eq2.6}
Y_{2,0}(M,[g_0])=\bar Y_{2,0}(M,[g_0]).
\end{equation}
As mentioned in the Introduction, the achievement of $\bar Y_{2,0}(M,[g_0])$ in the cone $\mathcal C_2([g_0])$ was proved in \cite{STW, GWPisa}; in view of \eqref{eq2.6}, the corresponding minimizer also achieves $Y_{2,0}(M,[g_0])$.

\section{The convergence of flow \eqref{eq2.1} }

In this section we prove  the $C^2$ estimates and the preservation of the positivity of scalar curvature of  flow \eqref{eq2.1}. Both imply its convergence.
\subsection{\texorpdfstring{$C^2$}{C2} estimates for flow (\ref{eq2.1})}\label{C2estimates-nonconformal-flow}
In this subsection we establish a priori estimates for the flow \eqref{eq2.1}. Local estimates for this class of fully nonlinear conformal equations were first obtained in \cite{GWIMRN}.

We consider \eqref{eq2.1} with initial metric $g(0)=g_1\in {\cal C}_{1}([g_0])$. By Lemma \ref{lem2.3}, the equation is parabolic as long as $r_\e(g)>0$, which holds under our assumptions. Short-time existence follows from a standard fixed-point argument (see Appendix~A in \cite{GWAdv}). We now prove gradient and second derivative  estimates for this flow.

\begin{thm}\label{thmlocal1}
Assume that $n>4$ and $g(0)=g_1\in {\cal C}_{1}([g_0])$. Let $u$ be a solution of \eqref{eq2.1} in a geodesic ball $B_R\times[0,T]$ for $T<T^*$ and $R<\tau_0$, the injectivity radius of $M$. Then there is a constant
$C$ depending only on $(B_R,g_0)$ and independent of $T$ such that, for any $(x,t)\in B_{R/2}\times[0,T]$,
\begin{equation}\label{eq_thm1.1}
|\n u|^2+|\n^2 u|\le C\bigl(1+r_\e(g)e^{-(2-\e)\inf_{B_R\times[0,T]}u}\bigr).
\end{equation}
\end{thm}

	%\noindent{\it Proof of Theorem \ref{thmlocal1}.}
\begin{proof}
The proof closely follows \cite{GWIMRN, GLWJDG}. Throughout the argument, $C$ (resp.~$c$) denotes a positive constant independent of $T$ that may change from line to line.
	Define
	\[
	\nu:=r_\e(g)e^{-(4-2\e)u},
	\qquad
	F(W,u):=\frac{\s_2(W)-\nu}{\s_1(W)}.
	\]
	Set
	\beq \ba{rcl}\ds\vs
	(F^{ij}(W,u))&:=&\ds \left(\frac{\partial F}{\partial w_{ij}}(W)\right)\\
	&=&\ds\vs \left(\frac{\s_1(W)T^{ij}-\s_2(W)\delta^{ij} +\nu\delta^{ij}}{\s_1^2(W)}\right) \ea \eeq
	where $(T^{ij})=(\s_1(W)\delta^{ij}-w^{ij})$ is the first Newton transformation associated with $W$, and $\delta^{ij}$ denotes the Kronecker symbol.
	By assumption, $\nu>0$. In view of Lemma \ref{lem2.3}, the matrix $(F^{ij})$ is positive definite and $F$ is concave for $W\in \Gamma_1^+$. Moreover,
	\begin{equation}\label{concave}
	\sum_{ijkl}\frac {\p^2\bigl( F(W,u)\bigr)}{\p w_{ij}\p w_{kl}} r_{ij}r_{kl} \le -2\frac{\nu (\sum_{i}r_{ii})^2}{ \s_1^3(W)}.
	\end{equation}
	Let $S(TM)$ denote the unit tangent bundle of $M$ with respect to the background metric $g_0$. Define a function $\tilde G:S(TM)\times [0,T]\to\R$ by
	\beq
	\tilde G(e,t)= (\n^2 u+|\n u|^2 g_0)(e,e).
	\eeq
	Without loss of generality, we assume $R=1$. Let $ \rho\in C^\infty_0(B_1)$ be a cut-off
	function defined as in \cite{GWIMRN} such that
	\begin{equation}\begin{array}{rcll}\label{eq4.1}
			\vs \ds \rho & \ge & 0, & \hbox{ in } B_1,\\
			\vs\ds \rho & =& 1, &\hbox{ in } B_{1/2},\\
			\vs\ds |\n \rho (x)|& \le &2 b_0 \,\rho^{1/2}(x), & \hbox{ in } B_1,\\
			|\n^2 \rho| & \le  & b_0,   & \hbox{ in } B_1.
	\end{array}\end{equation}
	Here $b_0>1$ is a constant. Since $e^{-2u}g_0\in {\cal C}_1$, it suffices to obtain an upper bound for
	$(\n^2 u+|\n u|^2 g_0)(e,e)$, uniformly for all $e\in S(TM)$ and all $t\in[0,T]$.
	To this end, set
	\[
	G(e,t):=\rho(x)\,\tilde G(e,t).
	\]
	Choose $(e_1,t_0)\in S(T_{x_0}M)\times(0,T]$ such that
	\beqr \label{eq4.2}
	\vs \ds G(e_1,t_0)=\max_{S(TM)\times [0, T]} G(e,t).\eeqr
	We may further assume that
	\beqr \label{eq4.4}
	\ds G(e_1,t_0)>n\max_{B_1} \s_1(g_0).
	\eeqr
	Let $(e_1,\dots,e_n)$ be an orthonormal basis at $(x_0,t_0)$.

	Choose normal coordinates around $x_0$ so that, at $x_0$, we have
	\[
	\frac{\p}{\p x_1}=e_1.
	\]
	Then $(x_0,t_0)$ is a maximum point of the scalar function
	$$G(x,t):=\rho(x)\bigl(u_{11}+|\nabla u|^2\bigr)(x,t)$$
	on $M\times[0,T]$. Consequently, at $(x_0,t_0)$ we obtain
	\begin{eqnarray}
		\label{y1}
		0&\le& G_t=\rho\Bigl(u_{11t}+2\sum_{l}u_lu_{lt}\Bigr),\\
		\label{y2}
		0&=&G_j=\frac{ \rho_j }{\rho} G + \rho\Bigl(u_{11j}+2\sum_{l\ge 1}u_l u_{lj}\Bigr),
		\quad \hbox{ for any } j,\\
		0&\ge&(G_{ij})=\ds \left(\frac{ \rho\rho_{ij}-2\rho_i\rho_j}{\rho^2} G
		+ \rho\Bigl(u_{11ij}+\sum_{l\ge 1}\bigl(2u_{li}u_{lj}+2u_lu_{lij}\bigr)\Bigr)\right).
	\end{eqnarray}
	Recall that $(F^{ij})$ is positive definite. Hence,
	\beq\label{x-2}
	\ba{rcl}
	0&\ge&\ds\vs \sum_{i,j\ge 1}F^{ij}G_{ij}-G_t\\
	&\ge&\ds\vs\sum_{i,j\ge 1}F^{ij}\frac{ \rho \rho_{ij}-2\rho_i\rho_j}{\rho^2} G
	+\rho \sum_{i,j\ge 1} F^{ij}\Bigl(u_{11ij}+\sum_{l\ge 1}(2u_{li}u_{lj}+2u_lu_{lij})\Bigr)\\
	&&\ds-\rho\Bigl(u_{11t}+2\sum_{l\ge 1}u_lu_{lt}\Bigr). \ea \eeq
	First, by the properties of $\rho$,
	\beq \label{x-1}
	\sum_{i,j\ge 1}F^{ij} \frac{ \rho \rho_{ij}-2\rho_i\rho_j}{\rho^2} G
	\ge -C\sum_{i,j\ge 1}|F^{ij}|\,\frac{G}{\rho}.
	\eeq
	Moreover,
	\beq \label{x-3} \ba{rcl}
	\ds\sum_{i,j\ge 1} |F^{ij}|  & \ge &\ds\vs\sum_{i} F^{ii}  \\
	&=& \ds \vs \left( n-1-\frac {n\s_2(W)}{\s_1^2(W)}\right)+ \frac{n\nu}{\s_1^2(W)}
	\ge \frac{1}{n} \sum_{i,j\ge 1} |F^{ij}|, \ea \eeq
	since $\{F^{ij}\}$ is positive definite. In particular,
	\beq\label{x-4}
	\sum_i F^{ii} \ge \frac {n-1}{2}+\frac{n\nu}{\s_1^2(W)}\ge 1+\frac{3\nu}{\s_1^2(W)}.
	\eeq
	Using the commutation identities
	\begin{equation}\label{0x1}
		u_{kij}=u_{ijk}+\sum_m R_{mikj}u_m,
	\end{equation}
	\begin{equation}\label{0x2}
		u_{kkij}=u_{ijkk}+\sum_m \bigl(2R_{mikj}u_{mk}-Ric_{mj}u_{mi}-Ric_{mi}u_{mj}-Ric_{mi,j}u_m+R_{mikj,k}u_m\bigr),
	\end{equation}
	and
	\begin{equation}\label{0x3}
		\bigl(\sum_l u_l^2\bigr)_{11}=2\sum_l\bigl(u_{11l}u_l+u_{1l}^2\bigr)+O(|\n u|^2),
	\end{equation}
	we obtain
	\beq\label{x1} \ba{rcl}
	\ds \sum_{i,j\ge 1} F^{ij} u_{11ij}
	&\ge & \ds\vs \sum_{i,j\ge 1} F^{ij} \left(w_{ij,11}-(u_{11})_iu_j -u_i (u_{11})_j\right)\\
	&&\ds\vs+\sum_{i,j,l\ge 1}F^{ij}(u_{1l}^2 + u_{11l}u_l)(g_0)_{ij}\\
	& & \ds\vs -2\sum_{i,j\ge 1} F^{ij}u_{i1}u_{j1}-
	C(1+|\n^2 u|+|\n u|^2)\sum_{i,j\ge 1} |F^{ij}|
	\ea\eeq and \beq\label{x2}\ba{rcl} \ds\vs \sum_{i,j, l} F^{ij}
	u_{l}u_{lij}
	&\ge & \ds \sum_{i,j, l} F^{ij} u_{l}w_{ij,l}-\sum_{i,j, l} F^{ij}(u_lu_{il}u_j+u_lu_iu_{jl}) \\
	&&\ds +\frac1 2\sum_{i,j} F^{ij}\langle \n u, \n(|\n u|^2)\rangle (g_0)_{ij} -C(1+|\n u|^2)\sum_{i,j\ge 1} |F^{ij}|.
	\ea\eeq Combining (\ref{x1}) and (\ref{x2}), we deduce
	\beq\label{x3}\ba{rcl} & &\ds\vs \sum_{i,j\ge 1}
	F^{ij}(u_{11ij}+2\sum_{l\ge 1}(u_{li}u_{lj}
	+ u_l u_{lij}))\\
	&\ge &\ds\vs \sum_{i,j\ge 1} F^{ij} (w_{ij,11}+2\sum_{l\ge 1} w_{ij,l}u_l) +2\sum_{i,j\ge 1} F^{ij} \sum_{l\ge 2}u_{li} u_{lj}\\
	& &\ds\vs+\sum_{i,j,l\ge 1} u_{1l}^2 F^{ij}(g_0)_{ij} -\sum_{i,j} F^{ij}\left[(u_{11}+ |\n u|^2)_iu_j+ u_i(u_{11}+ |\n u|^2)_j\right.\\
	&&\ds\vs \left.-\langle \n u, \n(u_{11}+ |\n u|^2)\rangle (g_0)_{ij}\right]
	-C(1+|\n^2 u|+|\n u|^2)\sum_{i,j\ge 1} |F^{ij}|\\
	&\ge&\ds \vs \sum_{i,j} F^{ij} (w_{ij,11}+2\sum_l w_{ij,l}u_l)+u_{11}^2\sum_{i,j} F^{ij}(g_0)_{ij}\\
	&&\ds \vs + \sum_{i,j} F^{ij}\left( \rho_iu_j+\rho_ju_i
	-\langle\n\rho,\n u\rangle (g_0)_{ij} \right)
	\frac{G}{\rho^2}\\
	&&\ds\vs-C(1+|\n^2 u|+|\n u|^2)\sum_{i,j\ge 1} |F^{ij}|.
	\ea\eeq
	In the last inequality we have used (\ref{y2}).
	Now, we want to estimate $ \sum_{i,j,l} F^{ij} w_{ij,l}u_l$
	and $\sum_{i,j} F^{ij} w_{ij,11}$. By differentiating $F$ we get
	\beq\label{eq4.11}\ba{llll}
	\ds \sum_l F_lu_l= \sum_{i,j,l}F^{ij} w_{ij,l}u_l+\sum_{l} \frac{\p F}{\p u
	}u_l^2&=\ds\sum_{i,j,l} F^{ij} w_{ij,l}u_l+\sum_{l} \frac{(4-2\e)r_\e(g)e^{-(4-2\e)u}
		u_l^2}{\s_1(W) }. \ea\eeq
	By differentiating $F$ twice  and
	using the concavity  (\ref{concave}) of $F$ in $W$, we have
	\beq\label{eq4.12x}\ba{rcl} \ds \sum_{i,j} F^{ij} w_{ij,11}&=& \ds
	F_{11}-\sum_{i,j,k,m}\frac{\p^2 F}{\p w_{ij}\p w_{km}}
	w_{ij,1}w_{km,1}\\
	&&-\ds 2\sum_{i,j}\frac{\p^2 F}{\p w_{ij}\p u}
	w_{ij,1}u_1-\frac{\p^2 F}{\p^2 u}
	u_{1}^2-\frac{\p F}{\p u}
	u_{11}\\
	&\ge&\ds F_{11}+\frac{2\nu (\sum_i
		w_{ii,1})^2}{(\s_1(W))^3}+\frac{2(4-2\e)r_\e(g)e^{-(4-2\e)u} (\sum_i
		w_{ii,1})u_1}{(\s_1(W))^2} \\
	&&\ds+
	\frac{(4-2\e)^2r_\e(g)e^{-(4-2\e)u} u_1^2}{\s_1(W)}- \frac{(4-2\e)r_\e(g)e^{-(4-2\e)u} u_{11}}{\s_1(W)}\\
	&\ge&\ds F_{11}+ \frac{(4-2\e)^2r_\e(g)e^{-(4-2\e)u} u_1^2}{2\s_1(W)}-
	\frac{(4-2\e)r_\e(g)e^{-(4-2\e)u} u_{11}}{\s_1(W)}. \ea \eeq These
	estimates give 
    \beq\label{eq4.12}\ba{rcl} 
    &&\ds 
    \sum_{i,j\ge 1}
	F^{ij}
	(w_{ij,11}+2\sum_{l\ge 1} w_{ij,l}u_l)\\
    &\ge& F_{11} +2 \sum_l F_lu_l -\frac{(4-2\e)r_\e(g)e^{-(4-2\e)u} u_{11}}{\s_1(W)}-\sum_{l=1}^n
	\frac{2(4-2\e) \nu u_l^2}{\s_1(W) }. \ea\eeq 
    Recall our flow
	%(\ref{eq2.1}) that 
    \beq\label{eq4.13} F=u_t-s_\e(g). \eeq
	Hence we have
	\beq\label{eq4.14} F_{11}=u_{11t}, \eeq \beq\label{eq4.15}
	F_{l}=u_{lt},\quad\;\forall\, l=1,\cdots,n. \eeq
	Gathering %(\ref{eq4.5}), (\ref{eq4.10}),
	\eqref{x-2}, \eqref{x-1}, \eqref{x-3}, \eqref{x3}, \eqref{eq4.12}, \eqref{eq4.14} and \eqref{eq4.15}, we obtain \beq \label{eq4.16}
	\ba{rcl} 0 &\geq& \ds\vs -C\left(\sum_{i,j}
	\left|F^{ij}\right|\right) \frac{G}{\rho}+\rho\left(\sum_i F^{ii}
	\right)u_{11}^2  \ds\vs -C
	\rho\left(\sum_{i,j} \left|F^{ij}\right|\right) (1+|\n u|^2+|\n^2 u|)\\
	&& \ds\vs + \sum_{i,j} F^{ij}\left( \rho_iu_j+\rho_ju_i
	-\langle\n\rho,\n u\rangle (g_0)_{ij} \right)
	\frac{G}{\rho} \ds -\rho \frac{(4-2\e)r_\e(g)e^{-(4-2\e)u} }{\s_1(W)}(u_{11}+ 2\sum_{l=1}^n
	u_l^2). \ea \eeq
	Since $W\in\Gamma_1^+$, we have
	$$u_{11}(x_0,t_0)\ge \frac{1}{20}|\nabla u|^2(x_0,t_0),$$
	and hence
	$$G(x_0,t_0)\le 21\rho(x_0)u_{11}(x_0,t_0)$$
	(see (44) in \cite{GLWJDG}). Multiplying \eqref{eq4.16} by $\rho$ yields
	\beq\label{estimate1}
	0\ge \sum_{i} F^{ii}\Bigl(-C G +(\tfrac{G}{21})^2-CG^{3/2}\Bigr)
	-8\rho e^{-2u} r_\e(g) e^{-(2-2\e)u}\frac{G}{\s_1(W)}.
	\eeq
	If $\frac{G}{\s_1(W)}\ge 2352=16\times (21)^2/3$, then by \eqref{x-4},
	\[
	\frac12 \sum_{i} F^{ii}\Bigl(\frac{G}{21}\Bigr)^2
	-8 r_\e(g)\rho e^{-(4-2\e)u}\frac{G}{\s_1(W)}
	\ge \nu\left(\frac{G^2}{294\s_1^2(W)}-8\rho\frac{G}{\s_1(W)}\right)\ge 0.
	\]
	Together with \eqref{estimate1}, this implies
	\[
	0\ge \sum_{i} F^{ii}\Bigl(-C G +\frac12\bigl(\frac{G}{21}\bigr)^2-CG^{3/2}\Bigr),
	\]
	and hence $G(x_0,t_0)\le C$. This gives the desired estimate.
	If instead $\frac{G}{\s_1(W)}<2352$, the desired bound follows from \eqref{estimate1} and Lemma \ref{lem2.3}.
\end{proof}

\begin{cor}\label{thmglobal1}
Under the assumptions of Theorem \ref{thmlocal1}, there exists a constant $C$, depending only on $g_0$ (and independent of $T$), such that for all $t\in[0,T]$,
\begin{equation}
\|u(t,\cdot)\|_{C^2(M)}\le C.
\end{equation}
\end{cor}

\begin{proof}
By Lemma~\ref{lem2.2}, after renormalizing along the flow we may assume
\begin{equation}\label{volumebound}
\ds\int_M e^{2\e u}\,dvol(g)\equiv 1.
\end{equation}
Moreover, Lemma \ref{lem2.2} implies $0<r_\e(g(t))\le r_\e(g_0)$ for all $t\in[0,T]$.

\noindent{\bf Claim.} There exists $C>0$, independent of $T\in[0,T^*)$, such that for all $t\in[0,T]$ and $x\in M$,
\begin{equation}\label{eq6.7}
 u(t,x)\ge -C.
\end{equation}
Set $m(t):=\min_{(s,x)\in [0,t]\times M}u(s,x)$ and choose $s_t\in [0,T],\,x_t\in M$ so that $u(s_t,x_t)=m(t)$. Suppose, for contradiction, that there exists a sequence $t_n\to T$ with $m(t_n)\to-\infty$. Also there exists $s_n\in [0,t_n]$ such that $m(t_n)= u(s_n,x_{s_n})=m(s_n)$. 
Applying Theorem \ref{thmlocal1} at time $s_n$ yields, for all $x\in M$,
\[
|\nabla u(s_n,x)|^2\le C e^{-(2-\e)m(s_n)}.
\]
Therefore, for all $x\in B\bigl(x_{s_n},\,e^{(1-\e/2)m(s_n)}\bigr)$,
\[
|u(s_n,x)-m(s_n)|\le C.
\]
Consequently,
\[
\int_M e^{2\e u(s_n,\cdot)}\,dvol(g_{s_n})
\ge C\int_{B\left(x_{s_n},\,e^{(1-\e/2)m(s_n)}\right)} e^{(-n+2\e)m(s_n)}\,dvol(g_0)
\ge C e^{-(n-4)\e m(s_n)/2}\to\infty,
\]
contradicting \eqref{volumebound}. This proves the claim.

Combining \eqref{eq6.7} with Theorem \ref{thmlocal1}, we obtain a constant $C>0$, independent of $T\in[0,T^*)$, such that for all $(t,x)\in[0,T]\times M$,
\begin{equation}\label{eq6.7.2}
|\nabla u(t,x)|+|\nabla^2 u(t,x)|\le C.
\end{equation}
Finally, using \eqref{volumebound} again yields a uniform bound for $u$ itself, and hence
\[
|u(t,x)|+|\nabla u(t,x)|+|\nabla^2 u(t,x)|\le C,
\qquad \forall (t,x)\in[0,T]\times M.
\]
This completes the proof.
\end{proof}

	\subsection{Uniform parabolicity}\label{Uniform parabolicity for non-conformal flow}
	
	We prove in this subsection that our flow (\ref{eq2.1}) preserves the positivity of the
	scalar curvature. 	
	\begin{prop}
		\label{positivity} Let $\e>0$ and assume that $Y_{\e}(M,[g_0])>0$. Let $u$ be a solution of \eqref{eq2.1}. Then there is a constant $C_0>0$, independent of $T\in[0, T^*)$, such that $\s_1(g(t))>C_0$ for any $t\in[0, T]$.
	\end{prop}
	
\begin{proof}
The proof is adapted from \cite{GWCrelle}. Throughout the argument, $C$ denotes a positive constant independent of $T$ that may change from line to line.
Recall
\[
\ba{l}
\vs\ds W=(w_{ij})=\Bigl(\nabla^2_{ij}u+u_i u_j-\frac{|\n u|^2}{2}(g_0)_{ij}+(S_{g_0})_{ij}\Bigr),\\
\ds \nu=r_\e(g)e^{-(4-2\e)u}.
\ea
\]
For $\kappa>0$ (to be chosen large), set
\[
G:=\frac{\s_2(W)-\nu}{\s_1(W)}-\kappa e^{-2u}.
\]
In view of \eqref{eq4.13}, this can be written as
$$G=u_t-\kappa e^{-2u}-s_\e(g(t)).$$
It follows from Lemma \ref{lem2.2}  and the assumption {$Y_{\varepsilon}(M,[g_0])>0$} that there exist positive constants $c_*$ and $c^*$ such that
\begin{equation}\label{upperandlowe bound of r}
 c_*\le r_\e(g)\le c^*.
\end{equation}

Assume that the minimum of $G$ on $M\times[0,T]$ is achieved at $(x_0,t_0)\in M\times(0,T]$. Choose normal coordinates at $x_0$ so that $W$ is diagonal at $(x_0,t_0)$.
Near $(x_0,t_0)$ we compute
\beq\label{v1}
\ba{rcl}
\ds\vs \frac{d}{dt}G
&=&\ds\sum_{i,j}F^{ij}(\nabla_g^2(u_t))_{ij}+2\kappa e^{-2u}u_t+\frac{(4-2\e)r_\e(g)e^{-(4-2\e)u}u_t}{\s_1(W)}-\frac{\alpha(t)}{\s_1(W)}\\
&=&\ds\vs\sum_{i,j}F^{ij}\Bigl((\nabla_g^2G)_{ij}+\kappa (\nabla_g^2(e^{-2u}))_{ij}\Bigr)
+2\kappa e^{-2u}u_t+\frac{(4-2\e)r_\e(g)e^{-(4-2\e)u}u_t}{\s_1(W)}-\frac{\alpha(t)}{\s_1(W)},
\ea
\eeq
where
\[
\alpha(t):=\frac{dr_\e(g)}{dt}e^{-(4-2\e)u}.
\]
As in the proof of Theorem \ref{thmlocal1},
\[
F^{ij}:=\frac{\p F}{\p w_{ij}}=\frac{(\s_1^2(W)-\s_2(W)+\nu)\delta^{ij}-\s_1(W)W^{ij}}{\s_1^2(W)}
\]
is positive definite.

\smallskip
\noindent{\bf Claim.} There exists $c_2>0$, independent of $T$ and $\kappa$, such that for all $t\in[0,T]$,
\begin{equation}\label{eq_add3}
\left|\frac{dr_\e(g)}{dt}(t)\right|\le c_2\left(1+\kappa+\frac{1}{\s_1(W)(x_0,t_0)}\right).
\end{equation}
Using \eqref{eq2.1}, \eqref{eq2.4}, and Corollary \ref{thmglobal1}, one has
\[
|s_\e(g)|\le c\left(1+\int_M\frac{1}{\s_1(W)(x,t)}\,dvol(g_0)\right),
\]
and therefore
\beq\label{eq_add}
\left|\frac{dr_\e(g)}{dt}(t)\right|\le c\left(1+\int_M\frac{1}{\s_1(W)(x,t)}\,dvol(g_0)\right).
\eeq
Since $(x_0,t_0)$ is the minimum point of $G$ on $M\times[0,T]$, for any $(x,t)\in M\times[0,T]$ we have
\beq\label{add_x}
\frac{\nu}{\s_1(W)}(x,t)
\le \frac{\s_2(W)}{\s_1(W)}(x,t)-\kappa e^{-2u(x,t)}-\frac{\s_2(W)-\nu}{\s_1(W)}(x_0,t_0)+\kappa e^{-2u(x_0,t_0)}.
\eeq
By Corollary \ref{thmglobal1}, the quantities $\s_1(W)$, $\s_2(W)$ and $e^{-2u}$ are bounded, and $\nu$ is bounded above and below by positive constants. Together with
$$\s_2(W)(x,t)\le \frac{n-1}{2n}\s_1^2(W)(x,t),$$
\eqref{add_x} implies that there exists $c_3>0$, independent of $T$ and $\kappa$, such that for all $(x,t)\in M\times[0,T]$,
\begin{equation}\label{eq_add6}
\frac{1}{\s_1(W)(x,t)}\le c_3\left(1+\kappa+\frac{1}{\s_1(W)(x_0,t_0)}\right).
\end{equation}
Combining \eqref{eq_add6} with \eqref{eq_add} yields \eqref{eq_add3}.

As a consequence, at $(x_0,t_0)$ we have
\begin{equation}\label{eq_add5}
\frac{|\alpha(t_0)|}{\s_1(W)}\le c\left(\frac{1+\kappa}{\s_1(W)}+\frac{1}{\s_1^2(W)}\right).
\end{equation}
	Since $(x_0,t_0)$ is a minimum point of $G$ on $M\times[0,T]$, at $(x_0,t_0)$ we have
	$$\ds\frac{dG}{dt}\le 0,\qquad G_l=0\ \forall l,$$
	and $(G_{ij})$ is non-negative definite (with respect to $g_0$). Note that
	\[
	(\n^2_g)_{ij}G
	=G_{ij}+u_iG_j+u_jG_i-\sum_l u_l G_l\,\d_{ij}
	=G_{ij}
	\quad\text{at }(x_0,t_0),
	\]
	where $G_j$ and $G_{ij}$ denote the first and second derivatives with respect to the background metric $g_0$.
	From the positivity of $\{F^{ij}\}$ and \eqref{v1}, we obtain
	\begin{equation}
		\label{5add1}\begin{array}{rcl}
			0& \ge &  \vs\ds G_t- \sum_{i,j} F^{ij} G_{ij} \\
			&&\\
			&\ge&\vs\ds \kappa\sum_{i,j}  F^{ij}
			\{(e^{-2u})_{ij}+u_i(e^{-2u})_{j}+u_j(e^{-2u})_{i}-\sum_l u_l(e^{-2u})_l\d_{ij}\}\\
			&&\vs\ds + 2\kappa e^{-2u}u_t+\frac{(4-2\e) r_\e(g) e^{-(4-2\e)u}u_t}{\s_1(W)}-\frac \alpha{\s_1(W)}\\
			&=&\vs\ds \kappa e^{-2u}\sum_{i,j}  F^{ij}\{- 2w_{ij}+2u_iu_j+2S(g_0)_{ij}+|\n u|^2\d_{ij}\}\\
			&&\vs\ds + 2\kappa e^{-2u}u_t+\frac{(4-2\e) r_\e(g) e^{-(4-2\e)u}u_t}{\s_1(W)}-\frac \alpha{\s_1(W)}\\
			&= & \ds \vs\kappa
			e^{-2u}\left(\frac{-2\s_2(W)-2\nu}{\s_1(W)}\right)  + 2\kappa e^{-2u}u_t+\frac{(4-2\e) r_\e e^{-(4-2\e)u}u_t}{\s_1(W)}\\
			&&\vs\ds+\kappa e^{-2u} \sum_{i,j}
			F^{ij}(2u_iu_j+2S(g_0)_{ij}+|\n u|^2\d_{ij})-\frac \alpha{\s_1(W)}.
	\end{array}\end{equation}
	Here we used
	$$\ds\sum_{i,j}F^{ij}w_{ij}=\frac{\s_2(W)+\nu}{\s_1(W)}.$$ 
	A direct computation yields
	\begin{equation}\label{5add2}
	\begin{array}{lll}
	\ds\sum_{i,j} F^{ij}S(g_0)_{ij}
	&=&\ds\vs \frac{(\s_1^2(W)-\s_2(W))\s_1(g_0)}{\s_1^2(W)}
	-\frac1{\s_1(W)}\sum_{i,j}W^{ij}S(g_0)_{ij}+\frac{\nu\s_1(g_0)}{\s_1^2(W)}.
	\end{array}
	\end{equation}
	Combining \eqref{eq_add6}, \eqref{eq_add5}, \eqref{5add1} and \eqref{5add2}, we obtain
	\begin{equation}\label{5add3}
	\begin{array}{rcl}
	0&\ge&\vs\ds G_t-\sum_{i,j}F^{ij}G_{ij}\\
	&\ge&\vs\ds \kappa e^{-2u}\left[\frac{-2\s_2(W)-2\nu}{\s_1(W)}+2\left(\frac{\s_2(W)-\nu}{\s_1(W)}+s_\e(g)\right)
	\right.\\
	&&\left.\ds\vs+\frac{2(\s_1^2(W)-\s_2(W))\s_1(g_0)}{\s_1^2(W)}-\frac2{\s_1(W)}\sum_{i,j}W^{ij}S(g_0)_{ij}+\frac{2\nu\s_1(g_0)}{\s_1^2(W)}\right]\\
	&&\ds -C\left(1+\frac{1+\kappa}{\s_1(W)}+\frac{1}{\s_1^2(W)}\right),
	\end{array}
	\end{equation}
	since $(F^{ij})$ is positive definite and $\kappa>0$. We use $O(1)$ to denote terms with a uniform bound. By Corollary \ref{thmglobal1}, $\|u\|_{C^2}$ is uniformly bounded; in particular $\s_2(W)=O(1)$ and
	$$\ds\sum_{i,j}W^{ij}S(g_0)_{ij}=O(1).$$
	Also $\s_1^2(W)-\s_2(W)\ge 0$. The only potentially harmful contribution in \eqref{5add3} is the term of order $\s_1(W)^{-2}$.

	Choose $\kappa\ge 1$ so large that
	\[
	\frac{\kappa\nu\s_1(g_0)e^{-2u}}{\s_1^2(W)}\ge \frac{C}{\s_1^2(W)}.
	\]
	Then \eqref{5add3} implies at $(x_0,t_0)$ that
	\[
	0 \ge  \frac{\kappa}{\s_1^2(W)}-c_4\left(\frac{\kappa+1}{\s_1(W)}+1\right)
	\]
	for some constant $c_4>0$ independent of $T$ and $\kappa$. Consequently, there exists $c_5>0$ (independent of $T$) such that
	\[
	\s_1(W)(x_0,t_0)\ge c_5.
	\]
	%which
	%implies for some positive constant $C_2'$ independent of $T$
	%\beq \label{5add4}
	%\frac{|\s_2(W)(x_0,t_0)|+\nu}{\s_1(W)(x_0,t_0)}+\kappa e^{-2u}<C_2'.
	%\eeq Since $(x_0,t_0)$ is the minimum of $F$ in $ M\times [0,T]$,
	%for any $(x,t)\in M\times [0,T]$ we have
	%\[
	%\frac{\s_2(W)(x,t)-\nu}{\s_1(W)(x,t)}- \kappa e^{-2u(x,t)}\ge -C_2'.
	%\]
	Hence, by \eqref{eq_add6},  there exists a constant $c_6>0$, independent of $T$, such that for all $(x,t)\in M\times[0,T]$,
	\[
	\s_1(W)(x,t)\ge c_6.
	\]
	This completes the proof of Proposition~\ref{positivity}.
\end{proof}

\subsection{Convergence of the flow}
We now establish convergence of flow \eqref{eq2.1}.

\begin{thm}\label{existence}
Let $\e>0$ and assume that $Y_{\e}(M,[g_0])>0$. Then flow \eqref{eq2.1} with initial metric $g_1\in {\cal C}_{1}([g_0])$ converges to a metric $g_\infty$ satisfying \eqref{eq2.5}. In particular, $g_\infty\in {\mathcal C}_2([g_0])$. Moreover, $Y_{\e}(M,[g_0])$ is achieved by a conformal metric $g_\e\in {\mathcal C}_2([g_0])$.
\end{thm}

\begin{proof}
With the uniform $C^2$ estimates (Corollary~\ref{thmglobal1}) and uniform parabolicity (Proposition~\ref{positivity}), the standard arguments of \cite{GWCrelle} yield long-time existence and convergence of the flow to a stationary solution of \eqref{eq2.1}, hence a solution of \eqref{eq2.5}. The inclusion $g_{\infty}\in\mathcal C_2([g_0])$ follows from \eqref{upperandlowe bound of r}.

Furthermore, under the normalization $\int_M e^{2\e u}\,dvol(g_0)=1$ (cf.~the proof of Corollary~\ref{thmglobal1}), the set of solutions to \eqref{eq2.5} is compact. Therefore $Y_{\e}(M,[g_0])$ is achieved by some conformal metric in $\mathcal C_2([g_0])$.
\end{proof}

\section{Proof of Theorem \ref{mainthm}}
We now prove Theorem \ref{mainthm}, starting with the following lemma.

\begin{lem}\label{mainthmbis}
Let $(M^n,g_0)$ be a compact Riemannian manifold with $g_0\in \mathcal C_1$  and $n\ge 5$. Assume that $Y_{2}([g_0])>0$. Then for any $\e\in(0,1)$,
\beq\label{eq0.3bis}
Y_{\e}([g_0])=\bar Y_{\e}([g_0]).
\eeq
Moreover, $Y_{\e}([g_0])$ is achieved by a conformal metric $g\in\mathcal C_2([g_0])$.
\end{lem}

\begin{proof}
Let $g=e^{-2u}g_0$ be a conformal metric with positive scalar curvature. By H\"older's inequality, for all $\e\in(0,1)$,
\[
\int_M e^{2\e u}\,dvol(g)
\le \left(\int_M 1\,dvol(g)\right)^{\frac{n-2\e}{n}}\left(\int_M 1\,dvol(g_0)\right)^{\frac{2\e}{n}}.
\]
Consequently,
\[
\tilde {\cal F}_{2,\e}(g)
\ge {\cal F}_{2}(g)\left(\int_M 1\,dvol(g_0)\right)^{-\frac{2\e(n-4)}{n(n-2\e)}}.
\]
Taking the infimum over ${\cal C}_1([g_0])$ yields
\begin{equation}\label{eq_ine1}
Y_\e([g_0])\ge Y_{2,0}([g_0])\left(\int_M 1\,dvol(g_0)\right)^{-\frac{2\e(n-4)}{n(n-2\e)}}>0.
\end{equation}
Therefore, Theorem \ref{existence} provides a metric $g_\e\in\mathcal C_2([g_0])$ achieving $Y_\e([g_0])$. Since $Y_{\e}([g_0])\le \bar Y_{\e}([g_0])$, we obtain \eqref{eq0.3bis}.
\end{proof}

\begin{lem}
\[
\lim_{\e\to 0}\bar Y_\e([g_0])=\lim_{\e\to 0}Y_\e([g_0])=\bar Y_{2,0}([g_0])=Y_{2,0}([g_0]).
\]
\end{lem}

\begin{proof}
Fix $g\in{\cal C}_1([g_0])$. Then
\[
\lim_{\e\to 0}\tilde{\cal F}_{2,\e}(g)=\tilde{\cal F}_{2,0}(g).
\]
Taking the infimum over ${\cal C}_1([g_0])$ gives
\[
Y_{2,0}([g_0])\le \liminf_{\e\to 0}Y_\e([g_0]).
\]
On the other hand, by \eqref{eq_ine1} we have $Y_\e([g_0])\ge Y_{2,0}([g_0])\cdot(1+o(1))$ as $\e\to 0$, and therefore
\[
\limsup_{\e\to 0}Y_\e([g_0])\ge Y_{2,0}([g_0]).
\]
Combining these inequalities yields
\[
\lim_{\e\to 0}Y_\e([g_0])=Y_{2,0}([g_0]).
\]
The corresponding statement for $\bar Y_\e([g_0])$ is proved in the same way, and the identity $Y_{2,0}([g_0])=\bar Y_{2,0}([g_0])$ follows from Lemma~\ref{mainthmbis}.
\end{proof}

\begin{proof}[Proof of Theorem~\ref{mainthm}]
It remains to show that if $\mathcal C_2([g_0])\neq\emptyset$, then \eqref{eq0.3twoinvariantequ} holds. Under this assumption, by \cite{GLWJDG} one has $Y_{2,1}(M,[g_0])>0$. It is also well known that $Y_{1,0}([g_0])>0$. For any $g\in\mathcal C_1([g_0])$ we estimate
\[
\begin{array}{rcl}
\ds\vs\frac{\int_M \s_2(g)\,dvol(g)}{\left(\int_M \s_0(g)\,dvol(g)\right)^{\frac{n-4}{n}}}
&=&\ds\frac{\int_M \s_2(g)\,dvol(g)}{\left(\int_M \s_1(g)\,dvol(g)\right)^{\frac{n-4}{n-2}}}\cdot
\frac{\left(\int_M \s_1(g)\,dvol(g)\right)^{\frac{n-4}{n-2}}}{\left(\int_M \s_0(g)\,dvol(g)\right)^{\frac{n-4}{n}}}\\
&\ge& Y_{2,1}([g_0])\left(Y_{1,0}([g_0])\right)^{\frac{n-4}{n-2}}.
\end{array}
\]
It follows that
\begin{equation}
Y_{2,0}([g_0])\ge Y_{2,1}([g_0])\left(Y_{1,0}([g_0])\right)^{\frac{n-4}{n-2}}>0.
\end{equation}
\end{proof}

\section{Examples}
\label{AppendixA}
In this section we present examples illustrating what can go wrong without the assumption of positive scalar curvature.

Let $g_0=g_{\mathbb{S}^{n}}=\frac{1}{1-s^{2}}ds^{2}+\left(1-s^{2}\right)g_{\mathbb{S}^{n-1}},\,\, s=x_{n+1}$.
Consider the metric $g_{\ell }=e^{-2\ell s^{2}}g_{\mathbb{S}^{n}}.$
Recall 
\[
|\nabla_{g_{0}}s|^{2}=1-s^{2}\text{ and }\nabla_{g_{0}}^{2}s=-sg_{0}.
\]
For a conformal change $g_{u}=e^{2u}g_{0}$, the Schouten tensor transforms
as 
\[
A_{g_{u}}=A_{g_{0}}-\nabla^{2}u+du\otimes du-\frac{1}{2}|\nabla_{g_{0}}u|^{2}g_{0}.
\]

Let
\[
u_{\ell }(s):=-\ell s^{2},\quad g_{\ell }:=e^{2u_{\ell }}g_{0}=e^{-2\ell s^{2}}g_{0}.
\]
We have
\[
\nabla_{g_{0}}u=u'\nabla_{g_{0}}s
,\quad 
|\nabla_{g_{0}}u|^{2}=\left(u^{\prime}\right)^{2}|\nabla s|^{2}=4\ell ^{2}s^{2}\left(1-s^{2}\right),
\]
\[
\nabla_{g_{0}}^{2}u=u''\nabla_{g_{0}}s\otimes\nabla_{g_{0}}s-su'g_{0}.
\]
The eigenvalues of $g_{0}^{-1}A_{g_{\ell }}$ are
\[
\begin{gathered}\lambda_{1}(s)=\frac{1}{2}+2\ell -4\ell s^{2}+2\ell ^{2}s^{2}\left(1-s^{2}\right),\\
\lambda_{2}(s)=\frac{1}{2}-2\ell s^{2}-2\ell ^{2}s^{2}\left(1-s^{2}\right),
\end{gathered}
\]
where $\lambda_{2}(s)$ is the eigenvalue with multiplicity $n-1$.
Hence
\[
\sigma_{2}(g_{0}^{-1}A_{g_{\ell }})=\frac{(n-1)(n-2)}{2}\lambda_{2}^{2}+(n-1)\lambda_{1}\lambda_{2}
\]
and 
\begin{align*}
\s_2(g_\ell )dv_{g_{\ell }}=\sigma_{2}\left(g_{\ell }^{-1}A_{g_{\ell }}\right)dv_{g_{\ell }}&=e^{(n-4)u_{\ell }}\sigma_{2}\left(g_{0}^{-1}A_{g_{\ell }}\right)dv_{g_{0}}\\
&=e^{-(n-4)\ell s^{2}}\left(\frac{(n-1)(n-2)}{2}\lambda_{2}^{2}+(n-1)\lambda_{1}\lambda_{2}\right)dv_{g_{0}}.
\end{align*}
Since
\[dv_{g_0}=
dv_{g_{\mathbb{S}^{n}}}=\left(1-s^{2}\right)^{\frac{n-2}{2}}dv_{g_{\mathbb{S}^{n-1}}}ds,
\]
we have 
\begin{align*}
 & \int_{\mathbb{S}^{n}}\sigma_{2}(g_\ell )dv_{g_{\ell }}\\
 =& \omega_{n-1}\int_{-1}^{1}e^{-(n-4)\ell s^{2}}\left(\frac{(n-1)(n-2)}{2}\lambda_{2}^{2}+(n-1)\lambda_{1}\lambda_{2}\right)\left(1-s^{2}\right)^{\frac{n-2}{2}}ds\\
= &2\omega_{n-1}\int_{0}^{1}e^{-(n-4)\ell s^{2}}\left(1-s^{2}\right)^{\frac{n-2}{2}}\left(\frac{1}{2}-2\ell s^{2}-2\ell ^{2}s^{2}\left(1-s^{2}\right)\right)\\
 &\left((n-1)\left(\frac{1}{2}+2\ell -4\ell s^{2}+2\ell ^{2}s^{2}\left(1-s^{2}\right)\right)+\frac{(n-1)(n-2)}{2}\left(\frac{1}{2}-2\ell s^{2}-2\ell ^{2}s^{2}\left(1-s^{2}\right)\right)\right)ds.
\end{align*}

%\subsection{\texorpdfstring{$n\ge 5$}{n >= 5}}
For $n\ge 5$,
Taking $\ell s^2=t^2$, for $\ell$ sufficiently large we obtain
\begin{align*}
&\int_{\mathbb{S}^{n}}\sigma_{2}(g_\ell )\,dv_{g_{\ell }}\\
=& 2\omega_{n-1} \ell ^{3/2} \int_0^{\sqrt{\ell }}(-2t^2)\bigl(2(n-1)-(n-1)(n-4)t^2\bigr)e^{-(n-4)t^2}\,dt+O({\sqrt{\ell }})\\
=&-(n-1)\omega_{n-1} \ell ^{3/2}\frac{\sqrt{\pi}}{2(n-4)^{3/2}}+O({\sqrt{\ell }})<0,
\end{align*}
where $\int_0^{\infty} t^2 e^{-a t^2}\,dt=\frac{\sqrt{\pi}}{4 a^{3 / 2}}$ and $\int_0^{\infty} t^4 e^{-a t^2}\,dt=\frac{3 \sqrt{\pi}}{8 a^{5 / 2}}$.

%Taking $L=2000$ and $n=5$, using Matlab, \[\int_{\mathbb{S}^{5}}\sigma_{2}\left(g_{L}^{-1}A_{g_{L}}\right)dv_{g_{L}}\cong-1027055.2985998045<0.\]
%For $n=4$, $$\int_{\mathbb{S}^{4}}\sigma_{2}\left(g_{L}^{-1}A_{g_{L}}\right)dv_{g_{L}}=4\pi^2.$$

For the scalar curvature, when $\ell$ is large,
\begin{align*}
   \frac{1}{2(n-1)} \int_{\mathbb{S}^n}R_{g_\ell } dv_{g_\ell }
    =&2\omega_{n-1}\int_0^1
e^{-(n-2)\ell s^2}(1-s^2)^{\frac{n-2}{2}}\\
&\left(\frac{n}{2}+2\ell -2\ell (n+1)s^2+(4-2n)\ell ^2s^2(1-s^2)\right)ds\\
=& 4\omega_{n-1}\sqrt{\ell }\int_0^{\infty}e^{-(n-2)t^2}(1+(2-n)t^2)dt+O(\frac{1}{\sqrt{\ell }})\\
=&\omega_{n-1}\sqrt{\ell }\frac{\sqrt{\pi}}{\sqrt{n-2}}+O(\frac{1}{\sqrt{\ell }}).
\end{align*}
For the volume, when $\ell$ is large,
\begin{align*}\label{vol for large \ell }
\mathrm{Vol}(g_\ell )=&\omega_{n-1}\int_{-1}^1 e^{-n\ell s^2}(1-s^2)^{\frac{n-2}{2}}\,ds\\
=&2\frac{\omega_{n-1}}{\sqrt{\ell }}\int_0^{\sqrt{\ell }}e^{-nt^2}\left(1-\frac{t^2}{\ell }\right)^{\frac{n-2}{2}}\,dt
=\frac{\omega_{n-1}\sqrt{\pi}}{\sqrt{n\ell }}+O(\ell^{-3/2}).
\end{align*}
Here we used $\int_{-\infty}^{\infty} e^{-as^2}\,ds=\sqrt{\frac{\pi}{a}}$ for $a>0$.

Consequently,
\[\frac{\int_{\mathbb{S}^{n}}\sigma_{2}(g_\ell )\,dv_{g_{\ell }}}{\mathrm{Vol}(g_\ell )^{\frac{n-4}{n}}}\rightarrow -\infty, \quad \text{as}\quad  \ell \rightarrow +\infty,\]
and
\[\frac{\int_{\mathbb{S}^{n}}\sigma_{2}(g_\ell )\,dv_{g_{\ell }}}{\left(\int_{\mathbb{S}^n}R_{g_\ell }\,dv_{g_\ell }\right)^{\frac{n-4}{n-2}}}\rightarrow -\infty, \quad \text{as}\quad  \ell \rightarrow +\infty.\]
This example also indicates that, without the assumption of positive scalar curvature, the following infima cannot be achieved:
\[\inf_{g\in[g_0]}\frac{\int_{\mathbb{S}^{n}}\sigma_{2}(g)\, dv_{g}}{\mathrm{Vol}(g)^{\frac{n-4}{n}}},\quad \inf_{g\in[g_0]}\frac{\int_{\mathbb{S}^{n}}\sigma_{2}(g)\,dv_{g}}{\left(\int_{\mathbb{S}^n}R_g\,dv_{g}\right)^{\frac{n-4}{n-2}}}.\]
In fact, one can prove
\begin{prop} Let $n\ge 5$.  Any $n$-dimensional manifold  has \[\inf_{g\in[g_0]}\frac{\int_{M}\sigma_{2}\left(g\right)dv_{g}}{\left ( Vol(g) \right)^{\frac {n-4}n}}=-\infty, \quad 
\inf_{g\in[g_0]}\frac{\int_M\sigma_{2}\left(g\right)dv_{g}}{Vol(g)^{\frac{n-4}{n}}}=-\infty.
\]
\end{prop}
\begin{proof}
    The idea is to glue the family of metrics discussed above near a fixed point in $M$. We leave the proof to the interested reader.
\end{proof}

%%%%%%%%%%%%%%%%%%%%%%%%%%%%

\bibliography{bibforSigmakSobolev}
\bibliographystyle{amsplain}

\end{document}